\documentclass{article}
\usepackage[centertags]{amsmath}
\usepackage{amsfonts}
\usepackage{amssymb}
\usepackage{amsthm}
\usepackage{latexsym}
\usepackage{amsmath}
\usepackage{amsfonts}
\usepackage[margin=1in]{geometry}
\usepackage{graphicx}
\usepackage{xcolor}
\usepackage{epstopdf}
\usepackage{hyperref}

\newtheorem{theorem}{Theorem}[section]
\newtheorem{lemma}[theorem]{Lemma}

\newtheorem{remark}[theorem]{Remark}

\newcommand{\ep}{\epsilon}
\newcommand{\R}{\mathbb{R}}
\newcommand{\supp}{\mbox{\rm supp\hspace{1pt}}}
\newcommand{\curl}{{\rm curl\hspace{1pt}}}
\newcommand{\Div}{{\rm div\hspace{1pt}}}

\title{Uniqueness and increasing stability in electromagnetic inverse source problems}

\author{Victor Isakov\thanks{Department of Mathematics, Statistics, Physics, Wichita State University, Wichita, KS 67260-0033, USA. Email: victor.isakov@wichita.edu}
\qquad Jenn-Nan Wang\thanks{Institute of Applied Mathematical Sciences, NCTS, National Taiwan University, Taipei 106, Taiwan. Email:jnwang@math.ntu.edu.tw}}

\date{}

\begin{document}

\maketitle
\begin{abstract}
In this paper we study the uniqueness and the increasing stability in the inverse source problem for electromagnetic waves in homogeneous and inhomogeneous media from boundary data at multiple wave numbers. For the unique determination of sources, we consider inhomogeneous media and use tangential components of the electric field and magnetic field at the boundary of the reference domain. The proof relies on the Fourier transform with respect to the wave numbers and the unique continuation theorems. To study the increasing stability in the source identification, we consider homogeneous media and measure the absorbing data or the tangential component of the electric field at the boundary of the reference domain as additional data. By using the Fourier transform with respect to the wave numbers, explicit bounds for analytic continuation, Huygens' principle and bounds for initial boundary value problems, increasing (with larger wave numbers intervals) stability estimate is obtained.  
\end{abstract}

\section{Introduction}\label{se_intro}

The main theme of this paper is to investigate the inverse source problem for the Maxwell equations when the source is supported inside a bounded domain $\Omega$. We consider the scattering solution of the Maxwell equations due to the existence of the source. We measure suitable tangential components of  the electric field and the magnetic field on $\partial\Omega$ or a part of $\Omega$ to retrieve the information of the source. Inverse source problems have enormous applications in practice. For example, detection of submarines and of anomalies in various  industrial objects like material defects \cite{EV}, \cite{GS} can be regarded as recovery of acoustic sources from boundary measurements of the pressure. Other applications include antenna synthesis \cite{BLT2010}, biomedical imaging (magnetoencephalography and ultrasound tomography)  \cite{ABF}, fluorescent microscopy, and geophysics, in particular, to locating sources of earthquakes. 

Inverse source problems are linearisations of inverse problems of determining coefficients of partial differential equations.
From the boundary data for one single linear differential equation or system (that is, single wave number), it is not possible to find the source uniquely \cite[Ch.4]{I}. This non-uniqueness phenomenon also appears in the Maxwell equations due to the existence of non-radiating sources \cite{AM}, \cite{BC}. However, if we use the data collected for various wave numbers in $(0,K)$, the uniqueness can be restored, at least for divergence-free sources. For applications, the important issue is the stability of the source recovery. It is widely known that most of inverse problems for elliptic equations are ill-posed having a feature of logarithmic type stability estimates, which results in  a robust recovery of only few parameters describing the source and yields very low resolution numerically. In this work, we will show that for the Maxwell equations the stability of identifying divergence-free sources using absorbing boundary data on the whole $\partial\Omega$ with wave numbers in $(0,K)$ increases (getting nearly Lipschitz) when $K$ is getting large.

To describe  main results, we will use mostly standard notations. Let $\|\cdot\|_{(l)}$ denote the $H^l$ Sobolev norm of a scalar or a vector-valued functions, $\Omega$ be a bounded domain in ${\mathbb R}^3$ with connected  ${\mathbb R}^3\setminus \bar\Omega$ and the boundary  $\partial \Omega \in C^2$.  $C$ denotes a generic constant depending only on $\Omega, \ep_0, \mu_0$ whose value may vary from line to line. Consider the time-harmonic Maxwell equations in an inhomogeneous medium:
\begin{equation}
\label{maxeq}
\left\{
\begin{aligned}
& \curl E-i\omega \mu H=J_{\mu}\quad\mbox{in}\quad {\mathbf R}^3,\\
& \curl H+i\omega \epsilon E=J_\epsilon\quad\mbox{in}\quad {\mathbf R}^3,
\end{aligned}\right.
\end{equation}
where $E$, $H$ are electric and magnetic fields, $\omega>0$ is the wave number, $\epsilon$ and 
$\mu$ are $3\times 3$ real positive-definite matrices with time independent entries which are positive constants outside $\Omega$, i.e., for some $\ep_0>0, \mu_0>0$
\begin{equation}\label{pos}
\ep(x)=\ep_0 I_3\quad\mbox{and}\quad\mu(x)=\mu_0I_3,\quad x\in\R^3\setminus\bar\Omega,
\end{equation}
and $J_\ep, J_{\mu}$ are the (real vector valued)  electric and magnetic current densities that is assumed to be supported in $\Omega$
\begin{equation}\label{csupp}
\supp J_\ep,\; \supp J_\mu\subseteq\Omega.
\end{equation}
We are interested in the scattering solution for \eqref{maxeq}. In this case, $E,H$ are required to satisfy the Silver-M\"uller radiation condition:
\begin{equation}
\label{rad}
\lim_{|x|\to\infty}|x|(\sqrt{\mu_0}H\times\sigma-\sqrt{\ep_0}E)(x)=0,
\lim_{|x|\to\infty}|x|(\sqrt{\ep_0}E\times\sigma+\sqrt{\mu_0}H)(x)=0
\end{equation} 
where $\sigma=x/|x|$. One can show that for any $J_\ep, J_{\mu}\in H(\mbox{div},\Omega)$ satisfying \eqref{csupp} there exists a unique $(E,H)\in H(\curl,\mathbb  {R}^3)\times H(\curl, \mathbb{R}^3)$ satisfying \eqref{maxeq} and \eqref{rad}, where for any open set $D\subseteq\R^3$ we define $H(\mbox{div},D)=\{{\bf u}\in [L^2(D)]^3:\mbox{div}\,{\bf u}\in L^2(D)\}$ and
$H(\curl,D)=\{{\bf u}\in [L^2(D)]^3:\curl{\bf u}\in [L^2(D)]^3\}$. The corresponding graph norm  of $H(\curl,D)$ is defined by
\begin{equation}
\label{gnorm}
\|{\bf u}\|_{H(\curl,D)}=\left(\|{\bf u}\|^2_{[L^2(D)]^3}+\|\curl{\bf u}\|^2_{[L^2(D)]^3}\right)^{1/2}
\end{equation}
and ${H}_0(\curl,D)$ is the completion of $[C_0^\infty(D)]^3$ with respect to the  norm \eqref{gnorm}.

The first main result is uniqueness from the minimal data 
\begin{equation}
\label{Gamma}
E(\,,\omega)\times\nu,\; H(\, ,\omega)\times\nu\;\; \text{on}\;\; \Gamma \subset \partial \Omega,\; \text{for} \;\; K_\ast<\omega<K,
\end{equation}
where $0\le K_\ast<K$. 

\begin{theorem}\label{thm1}
Let $J_\mu, J_\ep\in H(\curl, \Omega)$ satisfy \eqref{csupp}. We further assume that $\ep,\mu\in C^2(\bar\Omega)$ and there exists a scalar function $\lambda(x)\in C^2(\bar\Omega)$ such that 
\begin{equation}\label{structure}
\ep(x)=\lambda(x)\mu(x),\quad x\in\Omega.
\end{equation}
Moreover, let $J_\ep, J_\mu$ be divergence-free, i.e.,
\begin{equation}
\label{divfree}
{\rm div} J_\ep=0,\;\;{\rm div} J_{\mu}=0\;\;\mbox{\rm in}\;\;\mathbb{R}^3.
\end{equation}
Then $J_\ep, J_{\mu}$ in \eqref{maxeq}, \eqref{rad} are uniquely determined by
\eqref{Gamma}. 
 \end{theorem}
 
 Observe that this result implies that $E(\, ,\omega)\times \nu$ on $\partial \Omega$ with $K_\ast<\omega <K$ under the conditions of theorem \ref{thm1} uniquely determines  
 $J_{\ep}, J_{\mu}$ on $\Omega$. Indeed, due to the uniqueness for 
 the exterior boundary value problem for the Maxwell system
 $E(\, ,\omega)\times \nu$ on $\partial \Omega$ uniquely determine
 $(E, H)$ on $\R^3\setminus \Omega$ and hence the data
 \eqref{Gamma} which implies uniqueness of $J_{\ep}, J_{\mu}$ .

The second main result of this paper is an improving stability of recovery of divergence-free sources $J_\ep, J_\mu$ from the \emph{absorbing boundary 
data (also called Leontovich condition)}
\begin{equation}
\label{DEH}
E(\cdot ,\omega)\times\nu-\alpha(\cdot)H_\tau(\cdot,\omega)\;\; \text{on}\;\; \partial \Omega,\;\; \text{for} \;\; 0<\omega<K,
\end{equation}
or the tangential component of the electric field 
\[
E(\, ,\omega)\times\nu\;\; \text{on}\;\; \partial \Omega,\;\; \text{for} \;\; 0<\omega<K,
\]
where $\nu$ is the unit outer normal of $\partial\Omega$ and $H_\tau=H-(H\cdot \nu)\nu$ is the tangential projection of $H$ on $\partial\Omega$. Here we assume that $\alpha(x)\in L^\infty(\partial\Omega)$ and $\alpha(x)\ge c>0$ on $\partial\Omega$. The case of $\alpha\equiv 1$ corresponds to the Silver-M\"uller boundary condition \cite{BH93}. In the next result we assume $\ep=\ep_0, \mu=\mu_0$. 
\begin{theorem}
\label{thm2}
Assume that $1<K$, sources $J_\mu$, $J_\ep$ satisfy \eqref{csupp}, \eqref{divfree}, and
\begin{equation}\label{m0}
\|J_\ep\|^2_{(1)}(\Omega)+\|J_\mu\|^2_{(1)}(\Omega)\le M_1^2
\end{equation}
or
\begin{equation}\label{m1}
\|J_\ep\|^2_{(2)}(\Omega)+\|J_\mu\|^2_{(2)}(\Omega)\le M_2^2
\end{equation}
for some $M_0, M_1>0$. Then there exists $C$, depending on $\mbox{\rm diam}\,\Omega,\ep_0,\mu_0$, such that
 \begin{equation}
\| J_\ep\|^2_{(0)}(\Omega)+\| J_\mu\|^2_{(0)}(\Omega)\leq
C\left(\varepsilon_0^2 +
\frac{M_1^2}{1+K^{\frac 43} {\cal E}_0^{\frac 23}}\right), 
\label{stability}
\end{equation}
or
 \begin{equation}
\| J_\ep\|^2_{(0)}(\Omega)+\| J_\mu\|^2_{(0)}(\Omega)\leq
C\left(\varepsilon_1^2 +
\frac{M_2^2}{1+K^{\frac 43} {\cal E}_1^{\frac 23}}\right), 
\label{stability2}
\end{equation}
for all $ (E,H) \in [H^1(\Omega)]^6$ solving \eqref{maxeq}, \eqref{rad}
where
$$
\varepsilon_0^2=\int_0^K\|E( ,\omega)\times \nu\ -\alpha H_\tau(,\omega)\|^2_{(0)}(\partial\Omega) d\omega,\;\; {\cal E}_0=|\ln\varepsilon_0|,
$$
and
$$
\varepsilon_1^2=\int_0^K\|E( ,\omega)\times \nu\|^2_{(1)}(\partial\Omega) d\omega,\;\; {\cal E}_1=|\ln\varepsilon_1|.
$$
 \end{theorem}

Observe that the stability bound \eqref{stability} or \eqref{stability2} contain a Lipschitz stable part $C\varepsilon_0^2$ or $C\varepsilon_1^2$ and a conditional logarithmic stable part. This logarithmic part is natural and necessary since we deal with elliptic systems. However with growing $K$ logarithmic part is decreasing and the stable bound is dominated by the Lipschitz part. Before going further, we would like to point out that the divergence-free condition \eqref{divfree} in Theorem~\ref{thm1} and \ref{thm2} is not for the technical reason. It is necessary for the uniqueness of our inverse problem. To see this, let $\varphi, \psi\in C^1(\mathbb{R}^3)$ be supported in $\Omega$ and
$
E=\frac{\nabla\varphi}{i\omega},\;\; H=-\frac{\nabla\psi}{i\omega},
$
then $(E,H)$ satisfies \eqref{maxeq} with $\ep=\mu=1$ and $J_\ep=\nabla\varphi$, $J_\mu=\nabla\psi$.  Such examples provide with non uniqueness to the determination of the source $\nabla\varphi, \nabla\psi$ from $E(,\omega), H( ,\omega)$ given outside $\Omega$.

The determination of a source using multiple frequencies has received a lot of attention in recent years. For the Helmholtz equation, uniqueness and numerical results were obtained in \cite{EV}. First increasing stability results were presented in \cite{BLT2010} for some particular cases. These results were proved by direct spatial Fourier analysis methods. In \cite{CIL}, using a different method involving a temporal Fourier transform, sharp bounds of the analytic continuation to higher wave numbers, and exact observability bounds for associated hyperbolic equations, increasing stability bounds were derived for the three dimensional Helmholtz equation.  Later in \cite{EI} the methods and results of \cite{CIL} are extended to the more complicated case of the two dimensional Helmholtz equation. We would like to point out that in the works mentioned above one uses the complete Cauchy data on $\partial\Omega$ instead of Dirichlet-like data, which is much more realistic. For instance, the common measuring acoustical devise (microphone) registers only pressure, while in seismic one typically collects displacements. Those data only register the Dirichlet boundary value on $\partial\Omega$. It should be mentioned that  in \cite{LY} a spherical $\Omega$ was considered and there was a result on increasing stability from only Dirichlet data on $\partial\Omega$, but the used norm of the data was not the standard norm. It involved the operator of solution of the exterior Dirichlet problem. In the recent preprint \cite{BLZ2018}, some results similar to \cite{LY} are obtained for the elastic and electromagnetic waves. 

The idea in the proof of our increasing stability result in Theorem~\ref{thm2} is motivated by the recent paper by Entekhabi and the first author \cite{EI18}, where  increasing stability bounds are obtained for the acoustic and elastic waves using the most natural Sobolev norms of the Dirichlet type data on an arbitrary domain $\Omega$. 
As in \cite{CIL} and \cite{EI18}, in this work we use the Fourier transform in time to reduce our inverse source problem to identification of the initial data in the time-dependent Maxwell equations by data on the lateral boundary. We derive our increasing stability estimate by using sharp bounds of analytic continuation of the data from $(0,K)$ onto $(0,+\infty)$ given in \cite{CIL} and then subsequently utilized in \cite{EI}, \cite{LY}, \cite{BLZ2018}. A new idea introduced in \cite{EI18} is to make use of the Huygens's principle and known Sakamoto type energy bounds for the corresponding hyperbolic initial boundary value problem (backward in time). These techniques enable them to avoid a need in the complete Cauchy data on $\partial\Omega$ and in a direct use of the exact boundary controllability results. For time-dependent Maxwell equations in homogeneous media, the Huygens' principle is valid. On the other hand, in our problem, in addition to Sakamoto type energy bounds, we also need the regularity estimate for the Maxwell equations with absorbing boundary condition or the tangential component of the electric field on the lateral boundary \cite{CE}, \cite{E08}.

 

The rest of this paper is organized as follows. In Section \ref{sec2}, we will prove the uniqueness theorem, Theorem~\ref{thm1}. We prove the increasing stability in Section~\ref{sec3} and \ref{sec4}. In Section~\ref{sec3}, we use the methods of \cite{CIL}, \cite{EI18}, in particular bounds of the analytic continuation of the needed norms of the boundary data from $(0,K)$ onto a sector of the complex plane  $\omega=\omega_1+i\omega_2$, and use them and sharp bounds in \cite{CIL} of the harmonic measure of $(0,K)$ in this sector to derive explicit bounds of the analytic continuation of this norms from $(0,K)$ onto the real axis. In Section~\ref{sec4}, we use the Fourier transform in time to transform the source problem of the time-harmonic Maxwell equations to the time-dependent homogeneous Maxwell equations with initial conditions. The derivation of increasing stability relies on the quantitative analytic continuation established in Section~\ref{sec3}, the Huygens' principle for the Maxwell equations in homogeneous media, and the regularity estimates using boundary conditions.   

\section{Proof of uniqueness}\label{sec2}

We first show solvability of the direct scattering problem and analyticity of its solution with respect to the wave number $\omega$.

\begin{theorem}
\label{thm3}
Assume that \eqref{pos}, \eqref{csupp} are satisfied and $J_\ep, J_{\mu}\in H(\mbox{\rm div},\Omega)$. Then there is a unique solution $(E(\, ,\omega), H(\, ,\omega))\in [H_{loc}(\curl,\R^3)]^2$ to the scattering problem  \eqref{maxeq}, \eqref{rad}. This solution has an (complex) analytic with respect to
$\omega=\Re\omega+i\Im\omega$ continuation onto a neighbourhood of the quarter plane $\{0<\Re\omega, 0\leq \Im\omega\}$ which for $0<\Im\omega$ satisfies the equation \eqref{maxeq} and exponentially decays for large $|x|$:
\begin{equation}
\label{expdecay}
|E(x,\omega)|+ |H(x,\omega)|\leq C e^{-C^{-1}|x|}
\end{equation}
with some constant $C$ depending only on $E,H, \omega$.
 \end{theorem}

We first prove a unique result from boundary data. 
\begin{lemma}
\label{lem22}
Assume that $\ep$ and $\mu$ are $C^1(\R^3)$ positive-definite matrix-valued functions. Let $\tilde\Omega$ be a domain in ${\mathbf R}^3$. If $\omega\ne 0$, $\curl E-i\omega \mu H=\curl H+i\omega \epsilon E=0$ on $\tilde\Omega$, and $E\times\nu=H\times\nu=0$ on $\Gamma\subset \partial\tilde\Omega$, then
$E=H=0$ on $\tilde\Omega$. 
\end{lemma}

Before proving this lemma we remind that, as widely known,  the Maxwell equations are invariant under a change of coordinates. To be precise, let the coordinate transform $x\to x'$ and $J=(J_{kl})$ with $J_{kl}=\partial x'_k/\partial x_l$ be the associated Jacobian matrix.  Then in the new coordinates $x'$, we have
\begin{equation*}
\left\{
\begin{aligned}
&\curl' H'=-i\omega\ep' E',\\
&\curl' E'=i\omega \mu' H',
\end{aligned}
\right.
\end{equation*} 
where
\[
E'=(J^T)^{-1}E,\; H'=(J^T)^{-1}H,\; \ep'=\frac{J\ep J^T}{\mbox{det}J},\; \mu'=\frac{J\mu J^T}{\mbox{det}J}.
\] 

We now prove Lemma~\ref{lem22}. 
\begin{proof}

First we observe that by elliptic regularity $(E,H)\in C^1(\tilde\Omega)$.  Let $P\in \Gamma$. We  claim that $E(P)=H(P)=0$. Not losing a generality we assume that $P$ is the origin and $\Gamma$ near $P$ is the graph of the function $x_3=\gamma(x_1,x_2)$ and moreover
$\partial_1\gamma(0)=\partial_2\gamma(0)=0$.  Let the change of coordinates $x\to x'$ be defined by $x_1'=x_1, x_2'=x_2, x_3'=x_3-\gamma(x_1,x_2)$ near $0$. Then we have
\[
J=\begin{pmatrix}1&0&0\\0&1&0\\-\partial_1\gamma&-\partial_2\gamma&1\end{pmatrix}\quad\mbox{and}\quad\mbox{det}J=1.
\]

 In the new coordinates the unit outer normal $\nu'=(0,0,-1)$, $E'\times\nu'=H'\times\nu'=0$  implies
\[
E_1'=E_2'=0,\quad H_1'=H_2'=0\quad\mbox{on}\quad\{x'_3=0\}. 
\]
In particular $\partial'_2H'_1(0)=\partial_1'H_2'(0)=0$, i.e., $\partial_1'H'_2(0)-\partial_2'H_1'(0)=0$. Next from the third component in the equation $curl' H'=i\omega\ep' E'$, we see that
\[
-i\omega\ep'_{33}(0)E_3'(0)=\partial_1'H'_2(0)-\partial_2'H_1'(0)=0
\]
and thus $E_3'(0)=0$. Transforming back to the original coordinates immediately gives $E(0)=0$.  Likewise, we can show that $H(0)=0$. In other words, we can prove that $E=H=0$ on $\Gamma$. We now apply the unique continuation result obtained in \cite{NW} to conclude that $E=H=0$ in $\tilde\Omega$.
\end{proof}

We also need the well-posedness and the regularity of the boundary value problem related to the Maxwell equations
\begin{equation}\label{maxbvp}
\left\{
\begin{aligned}
& \curl E^*-i\omega \mu H^*= J_{\mu}^*\quad\mbox{in}\quad B,\\
& \curl H^*+i\omega \epsilon E^*= J_{\epsilon}^*\quad\mbox{in}\quad B,\\
&E^*\times\nu=0\quad\mbox{on}\quad\partial B,
\end{aligned}\right.
\end{equation}
where $B$ is a ball and the source $J^*=(J_{\mu}^*,J_{\epsilon}^*)\in [L^2(B)]^6$. 

\begin{lemma}\label{lem23}
There exists a discrete set $${\mathcal T}=\{\cdots,\omega_{-2},\omega_{-1},\omega_1,\omega_2,\cdots\}$$ of nonzero real values, where $-\infty\leftarrow\cdots\le\omega_{-2}\le\omega_{-1}\le\omega_1\le\omega_2\le\cdots\rightarrow\infty$, such that for any $\omega\notin{\mathcal T}\cup\{0\}$ there is a unique solution $(E^*(\omega;J^*), H^*(\omega; J^*))$ to \eqref{maxbvp} and $(E^*(\omega; ), H^*(\omega; ))$ is a continuous linear operator from $[L^2(B)]^6$ into $H(\curl, B)^2$ which is analytic in $\omega\in{\mathbb C}\setminus({\mathcal T}\cup\{0\})$. Let $\{\omega_k(B)\}_{k=-\infty}^\infty$ and $\{\omega_k(B')\}_{k=-\infty}^\infty$ denote the discrete sets described above corresponding to balls $B$ and $B'$. Then if $B\subset B'$, then $\omega_k(B')<\omega_k(B)$ if $k>0$ and $\omega_k(B')>\omega_k(B)$ if $k<0$.
\end{lemma}
\begin{proof}
We first study the eigenvalue problem
\begin{equation}
\label{uvbvpp}
\left\{
\begin{aligned}
& \curl u-i\omega \mu v=0\quad\mbox{in}\quad B,\\
& \curl v+i\omega \epsilon u=0\quad\mbox{in}\quad B,\\
&  u \times \nu =0 \quad\mbox{on}\quad \partial B.
\end{aligned}\right.
\end{equation}
We can see that the eigenvalue problem \eqref{uvbvpp} is equivalent to the eigenvalue problem for $u$
\begin{equation}\label{ee}
\left\{
\begin{aligned}
&\curl(\mu^{-1}\curl u)=\omega^2 \epsilon u\quad\mbox{in}\quad B,\\
& u \times \nu =0 \quad\mbox{on}\quad \partial B.
\end{aligned}\right.
\end{equation}
For, it is clear that if $\omega\ne 0$ is an eigenvalue of \eqref{uvbvpp}, then $\omega^2$ is an eigenvalue of \eqref{ee}. Conversely, if $\omega^2$ is an eigenvalue of \eqref{ee} with eigenfunction $u$, then setting $v=\mu^{-1}\curl u/i\omega$ gives $\curl u-i\omega\mu v=0$ and $\curl v+i\omega \epsilon u=0$. 

The eigenvalue problem \eqref{ee} was completely analyzed in \cite{KH}. Recall from \cite[Theorem~4.34, page 193]{KH} that there exists an infinite number of positive eigenvalues $\omega^2_k$ with corresponding eigenfunction $u_k\in V_{0,\epsilon}$ to \eqref{ee}, where 
\[
V_{0,\epsilon}=\{u\in H_0(\curl, B) : (\epsilon u,\psi)_{L^2(B)}=0\; \mbox{for all}\; \psi\in H_0(\curl, B),\curl\psi=0\;\mbox{in}\; B\}.
\]
The eigenvalues $\{\omega^2_k>0\}$ have finite multiplicities and tend to infinity as $k\to\infty$. Moreover, $\{u_k\}_{k=1}^\infty$ form a complete orthonormal system of $(V_{0,\epsilon}, (\cdot,\cdot)_{\mu,\epsilon})$, where the inner product
\[
(u,v)_{\mu,\epsilon}=\int_B\mu^{-1}\curl u\cdot\curl\bar vdx+\int_B\epsilon u\cdot\bar vdx.
\]
Consequently, we have the formula
\[
\lambda_k:=\frac{1}{1+\omega_k^2}=(\epsilon u_k,u_k)_{L^2(B)}=\frac{(\epsilon u_k,u_k)_{L^2(B)}}{(u_k,u_k)_{\mu,\epsilon}}.
\]
Note that $\lambda_1\ge\lambda_2\ge\cdots\to 0$. It is not difficult to prove the following variational characterization of $\lambda_k$, that is,
\begin{equation}\label{vp}
\lambda_k=\max_{{\mathcal U}\subset V_{0,\epsilon}, \mbox{\scriptsize dim}\,{\mathcal U}= k}\;\min_{u\in{\mathcal U}, u\ne 0}\frac{(\epsilon u,u)_{L^2(B)}}{(u,u)_{\mu,\epsilon}},\;\; k=1,2,\cdots.
\end{equation}

An easy consequence of \eqref{vp} is that if $B\subset B'$, then 
\begin{equation}\label{wm}
\lambda_k(B)\le\lambda_k(B')\quad\mbox{for each}\;\; k,
\end{equation}
where $\lambda_k(B)=\frac{1}{1+\omega_k^2(B)}$, $\lambda_k(B')=\frac{1}{1+\omega_k^2(B')}$ and $\omega_k^2(B)$, $\omega^2_k(B')$ are eigenvalues of \eqref{ee} corresponding to $B$ and $B'$, respectively. We actually want to show that the strict monotonicity holds, i.e., for each $k$
\begin{equation}\label{mono}
\lambda_k(B)<\lambda_k(B')\;\;\mbox{if}\;\; B\subset B',
\end{equation}
which is equivalent to
\[
\omega^2_k(B)>\omega_k^2(B')\;\;\mbox{if}\;\; B\subset B'.
\]
We adopt the argument from \cite[Theorem~2.3]{W}. We will prove \eqref{mono} by contradiction. Assume that $\lambda_k(B)=\lambda_k(B')$. Since every $\lambda_k(B')$ has finite multiplicity and $\lambda_k(B')\to 0$, there exists $\lambda_n(B')<\lambda_k(B')$ for some $n$. We now partition $B'$ into $n$ balls satisfying
\[
B=B_1\subset B_2\subset\cdots\subset B_n=B'.
\]
Then \eqref{wm} implies
\[
\lambda_k(B)=\lambda_k(B_1)\le\lambda_k(B_2)\le\cdots\le\lambda_k(B_n)=\lambda_k(B').
\]
Denote $u_{k,j}$ the eigenfunction corresponding to $\lambda_k(B_j)$ with $\|u_{k,j}\|_{\mu,\epsilon}=1$, $j=1,2,\cdots,n$. To abuse the notation, we also use $u_{k,j}$ to denote the zero extension of $u_{k,j}$ originally defined on $B_j$ to $B'$. Still, we have $\|u_{k,j}\|_{\mu,\epsilon}=1$ with integral evaluated over $B'$. 

Now we would like to show that $\{u_{k,j}\}_{j=1}^n$ are linearly independent. Assume that $\sum_{j=1}^na_ju_{k,j}=0$ in $B'$, but $a_n\ne 0$, then $u_{k,n}=0$ in $B'\setminus B_{n-1}$. By the unique continuation property in Lemma~\ref{lem22}, we have that $u_{k,n}\equiv 0$ in $B'$, which is a contradiction. Other coefficients are treated similarly. Considering the subspace spanned by $\{u_{k,j}\}_{j=1}^n$ in the variational characterization of $\lambda_n(B')$ in \eqref{vp}, we obtain that $\lambda_k(B')\le\lambda_n(B')$, which is a contradiction.

To show the unique solvability of \eqref{maxbvp} for $\omega\not\in{\mathcal T}\cup\{0\}$, we consider the operator $L:{\mathcal D}(L)\to X:= [L^2(B)]^3\times [L^2(B)]^3$ given by
\[
L=\begin{pmatrix}0&-i\mu^{-1}\curl\\i\ep^{-1}\curl&0\end{pmatrix},
\]
where ${\mathcal D}(L)=H_0(\curl,B)\times H(\curl,B)$. It is not hard to check that $L$ is self-adjoint in $X$ with respect to the inner product
\[
\Big{\langle}\begin{pmatrix}{u}_1\\{ u}_2\end{pmatrix},\begin{pmatrix}{v}_1\\{ v}_2\end{pmatrix}\Big{\rangle}=\int_B(\mu{u}_1\cdot{v}_1+\ep{u}_2\cdot{v}_2),
\]
the range of $L$, $\mbox{\rm Ran}(L)$, is closed (see \cite[Corollary 8.10]{L86}).  Also, $X$ admits the orthogonal decomposition
\[
X=\mbox{\rm Ker}(L)\oplus\mbox{\rm Ran}(L).
\]
Let $P$ be the orthogonal projection of $X$ onto $\mbox{\rm Ran}(L)$. 

Let $(J_\mu^\ast, J_\ep^\ast)\in [L^2(B)]^3\times [L^2(B)]^3$, i.e., $F:=(-i\mu^{-1}J_{\mu}^*,-i\ep^{-1}J_{\mu}^*)\in X$, then to solve \eqref{maxbvp}, we consider
\begin{equation*}
(L-\omega)W=F,
\end{equation*}
where $W\in{\cal D}(L)$. If $\omega\not\in {\mathcal T}\cup\{0\}$, then $L-\omega$ is invertible. Hence the solution $W$ is given by
\[
W=(L-\omega)^{-1}PF-\omega^{-1}(I-P)F,
\]
for
\[
(L-\omega)W=(L-\omega)(L-\omega)^{-1}PF-(L-\omega)\omega^{-1}(I-P)F=(I-P)F+PF=F.
\]
Moreover, we can see that the solution $W$ is analytic in $\omega\in{\mathbb C}\setminus({\mathcal T}\cup\{0\})$. 
\end{proof}
\begin{remark}\label{rem11}
When $\ep=\ep_0I_3$ and $\mu=\mu_0I_3$, we denote the corresponding spectrum of $L$ by ${\mathcal T}_0$.
\end{remark}

We now prove Theorem~\ref{thm3}.
\begin{proof}
Let $0<\Re\omega$ and $0\le\Im\omega$. We first establish the uniqueness. In other words, we want to prove that if $(E,H)$ satisfies \eqref{maxeq} with $J_\ep\equiv J_\mu\equiv 0$ and \eqref{rad}, then $E=H=0$ in $\R^3$.  Let $\Omega_0$ be an open set containing $\bar\Omega$ with closure contained in a ball $B$. By the Gauss divergence theorem and the Maxwell equations \eqref{maxeq}, we have that
\[
\int_{\partial B}\nu\times E\cdot\bar H dS=\int_B(\curl E\cdot\bar H-E\cdot\curl\bar H)=\int_B(i\omega\mu H\cdot\bar H-i\bar\omega E\cdot\ep\bar E)
\]
and hence
\begin{equation}\label{iineq}
\Re\int_{\partial B}\nu\times E\cdot\bar H\, dS=-\Im\omega\int_B(\mu H\cdot\bar H+E\cdot\ep\bar E)dx\le 0.
\end{equation}
On the other hand, by \eqref{iineq}, we can see that
\begin{equation}\label{neg}
\begin{aligned}
&\Im(\omega\int_{\partial B}\nu\times E\cdot\curl\bar{E}\, dS)=\Im(\omega\int_{\partial B}\nu\times E\cdot(-i\bar\omega\bar H)\, dS)\\
=&|\omega|^2\Im(-i\int_{\partial B}\nu\times E\cdot\bar H\, dS)=-|\omega|^2\Re\int_{\partial B}\nu\times E\cdot\bar H\, dS\ge 0.
\end{aligned}
\end{equation}
In view of \eqref{neg}, using \cite[Theorem 4.17]{CK13}, we obtain that $E=0$ in $\R^3\setminus B$. Similarly, we can prove that $H=0$ in $\R^3\setminus B$. Combine this and Lemma~\ref{lem22} concludes that $E=H=0$ in $\R^3$.

We will prove the existence by the Lax-Phillips method. Let $\omega\in \{0<\Re\omega, 0\leq \Im\omega\}$ and $\Omega_0$ be an open set containing $\bar\Omega$. 
In view of the strict monotonicity of eigenvalues with respect to the domain proved in Lemma~\ref{lem23}, one can choose a ball $B, \bar \Omega_0 \subset B,$ so that $\omega\notin{\mathcal T}\cup{\mathcal T}_0$. Let $\phi$ be a  cut-off $C^{\infty}({\mathbb R}^3)$ function $\phi$  with $\phi=1$ on $\Omega$ and $\phi=0$ outside of $\Omega_0$. We look for a solution
\begin{equation}
\label{eq33}
\begin{pmatrix}E\\ H\end{pmatrix}=\begin{pmatrix}\Phi\\\Psi\end{pmatrix}-\phi\left(\begin{pmatrix}\Phi\\\Psi\end{pmatrix}-\begin{pmatrix}E^\ast\\H^\ast\end{pmatrix}\right)
\end{equation}
to system \eqref{maxeq}, where $\begin{pmatrix}E^\ast\\H^\ast\end{pmatrix}(\cdot,J^\ast)$ with $J^\ast=\begin{pmatrix}J^\ast_\mu\\J^\ast_\ep\end{pmatrix}$ being a solution to the boundary value problem
\begin{equation}
\label{uvbvp}
\left\{
\begin{aligned}
& \curl E^\ast-i\omega \mu H^\ast=J^\ast_\mu \quad\mbox{in}\quad B,\\
& \curl H^\ast +i\omega \epsilon E^\ast=J^\ast_\epsilon\quad\mbox{in}\quad B,
\\
&  E^\ast \times \nu =0 \quad\mbox{on}\quad \partial B,
\end{aligned}\right.
\end{equation}
and $J^\ast\in [H(\mbox{\rm div},B)]^2$ with $\supp J^\ast\subset B$ will be determined later. Moreover, $\begin{pmatrix}\Phi\\\Psi\end{pmatrix}$ is the solution to
\begin{equation}
\label{homo}
\left\{
\begin{aligned}
& \curl \Phi-i\omega \mu_0 \Psi=J^\ast_{\mu}\quad\mbox{in}\quad\R^3,\\
& \curl \Psi+i\omega \epsilon_0 \Phi=J^\ast_\epsilon\quad\mbox{in}\quad\R^3
\end{aligned}\right.
\end{equation}
satisfying the radiation condition
\begin{equation}\label{rad2}
\lim_{|x|\to\infty}|x|
(\sqrt{\ep_0}\Phi\times\sigma+\sqrt{\mu_0}\Psi) (x)=0,\;
\lim_{|x|\to\infty}|x|
(\sqrt{\mu_0}\Psi\times\sigma-\sqrt{\ep_0}\Phi) (x)=0.
\end{equation}
It is well known (see \cite{C}, p. 78, Theorem 2) that
\begin{equation}\label{intEH*}
\begin{aligned}
\Phi(x,\omega)&=\int_{\Omega}\frac{\exp({i\kappa|x-y|})}{4\pi|x-y|}(i\omega \mu_0 J^*_\ep(y)+\curl J^*_\mu(y))dy,\\
\Psi(x,\omega)&=\int_{\Omega}\frac{\exp({i\kappa|x-y|})}{4\pi|x-y|}(-i
\omega \ep_0 J^*_\mu(y)+\curl J^*_\ep(y))dy, \quad \kappa =\omega\sqrt{\ep_0\mu_0}.
\end{aligned}
\end{equation}

Since $\phi=1$ in $\Omega$, we have 
\[
J_\mu^\ast=J_\mu\quad\mbox{and}\quad J_\ep^\ast=J_\ep\quad\mbox{in}\quad\Omega.
\]
In $\R^3\setminus\bar\Omega$, we have
\[
\begin{aligned}
\curl E-i\omega\mu H&=\curl(\Phi-\phi(\Phi-E^\ast))-i\omega\mu_0(\Psi-\phi(\Psi-H^\ast))\\
&=\curl\Phi-\phi \curl(\Phi-E^\ast)-\nabla\phi\times(\Phi-E^\ast)-i\omega\mu_0(\Psi-\phi(\Psi-H^\ast))\\
&=J^\ast_\mu-\nabla\phi\times(\Phi-E^\ast)-\phi[\curl(\Phi-E^\ast)-i\omega\mu_0(\Psi-H^\ast)]\\
&=J^\ast_\mu-\nabla\phi\times(\Phi-E^\ast)
\end{aligned}
\]
and similarly
\[
\curl H+i\omega\ep E=J^\ast_\ep-\nabla\phi\times (\Psi-H^\ast).
\]
We introduce the operator
\[
A(\omega)J^\ast=\begin{pmatrix}-\nabla\phi\times(\Phi-E^\ast)\\-\nabla\phi\times(\Psi-H^\ast)\end{pmatrix}.
\]
Hence $\begin{pmatrix}E\\ H\end{pmatrix}$ is a scattering solution of \eqref{maxeq} iff
\begin{equation}\label{intop}
J=J^\ast+AJ^\ast.
\end{equation}
Note that $\supp AJ^\ast\subset\Omega_0\setminus\bar\Omega$. 

To prove the existence for \eqref{intop}, we first show that $I+A$ is Fredholm from $[H(\Div,B)]^2$ into itself. It follows from  \cite{AC14} that $\Phi(J^*), \Psi(J^*), E^\ast(J^*), H^\ast(J^*)$ are continuous linear operators from $[H(\Div,B)]^2$ into
$[H^1(B)]^6$. Moreover, by direct calculations,
$$
\Div(\nabla \phi \times  E^\ast) = - \nabla \phi \cdot \curl E^\ast=
- \nabla \phi \cdot (i\omega \mu_0 H^\ast +J^*_{\mu}),
$$
$$
\Div(\nabla \phi \times  \Phi) = - \nabla \phi \cdot  \curl \Phi=
- \nabla \phi \cdot (i\omega \mu_0 \Psi +J^*_{\mu}),
$$
due to \eqref{uvbvp}, since $\mu=\mu_0$ outside $\Omega$ and
$\nabla \phi=0$ on $\Omega$. 
Hence
$$
\Div(\nabla \phi \times  (\Phi - E^\ast)) = 
\nabla \phi \cdot (i\omega \mu_0 (H^\ast - \Psi)).
$$
Similarly,
$$
\Div(\nabla \phi \times  (\Psi - H^\ast)) = 
- \nabla \phi \cdot (i\omega \epsilon_0 (E^\ast - \Phi)).
$$
Summing up, $A$ is a continuous linear operator from
$[H(\Div,B)]^2$ into $[H^1(\Div,B)]^2$, where
$H^1(\Div,B)=\{u\in [H^1(B)]^3: \Div u \in H^1(B)\}$ with the natural norm. Since $H^1(B)$ is compactly embedded into $L^2(B)$, $A$ is compact from $[H(\Div,B)]^2$ into itself.

Now to establish the existence, it suffices to prove the injectivity of  $I+A$. Let $0=J^\ast+AJ^\ast$.  Since $J=0$, by the uniqueness
which was shown at the beginning of the proof, we have $E=H=0$ on $B$  thus from \eqref{eq33}
\[
\begin{pmatrix}\Phi\\\Psi\end{pmatrix}=\phi\left(\begin{pmatrix}\Phi\\\Psi\end{pmatrix}-\begin{pmatrix}E^\ast\\H^\ast\end{pmatrix}\right).
\]
Since $\phi=0$ on $B\setminus \Omega_0$ we have $\Phi=0$ on $\partial B$. Now from  \eqref{uvbvp}, \eqref{homo} we yield 
\begin{equation*}
\left\{
\begin{aligned}
& \curl (\Phi-E^\ast)-i\omega \mu(\Psi-H^\ast)=0\quad\mbox{in}\quad B,\\
& \curl (\Psi-H^\ast)+i\omega \epsilon(\Phi-E^\ast)=0\quad\mbox{in}\quad B,\\
&(\Phi-E^\ast)\times\nu=0\quad\mbox{on}\quad\partial B.
\end{aligned}\right.
\end{equation*}
By the choice of $B$ a solution to this boundary value problem is unique, we get $\Phi-E^\ast=\Psi-H^\ast=0$ on $B$ and hence $\Phi=\Psi=0$, so $AJ^*=0$ and from \eqref{intop} we conclude that  $J^\ast=0$. 

Summing up, the Fredholm operator $I+A(\omega)$ is injective, and hence has the inverse. Since $A(\omega)$ is analytic with respect to $\omega$, so is the inverse and therefore $J^*$. In view of the explicit representation formulas for $\Phi, \Psi$ in \eqref{homo} (see for example \cite[(47)]{C}) and the analyticity of $(E^\ast,H^\ast)$ in $\omega$ proved in Lemma~\ref{lem23}, the analyticity of $(E( ,\omega), H( ,\omega))$ follows. The exponential decay \eqref{expdecay} follows from \eqref{eq33}, \eqref{intEH*}.
\end{proof}

Now we prove Theorem~\ref{thm1}.

\begin{proof}

Due to the linearity it suffices to show that $E\times\nu=H\times\nu=0$ on $\Gamma$, $K_\ast < \omega < K$ implies that $J_\ep=J_{\mu}=0$. Let $(e,h)$ be a solution to the dynamical  initial boundary value problem:
\begin{align}
& \partial_t (\epsilon e)- \curl h = 0,\;
\partial_t (\mu h)+ \curl e=0 \quad \textrm{in} \;\; \mathbb{R}^3 \times (0,\infty), \label{maxdyn}\\
& e=-\sqrt{2\pi}\epsilon^{-1}J_\ep,\; h= -\sqrt{2\pi}{\mu}^{-1}J_{\mu}  \quad \textrm{on}\;\; \mathbb{R}^3 \times \{0\}. \label{initial}
\end{align}
Thanks to \eqref{divfree} and \eqref{maxdyn}, in addition to \eqref{initial}, we have the following compatibility conditions
\begin{equation}\label{compat}
\mbox{\rm div}(\ep e)=0,\;\;\mbox{\rm div}(\mu h)=0\;\;\mbox{in}\;\; \mathbb{R}^3 \times (0,\infty).
\end{equation}
As known, see for example \cite{Fried}, there is a unique solution $(e,h)\in L^{\infty}((0,T); [H(\curl,\mathbb{R}^3)]^6)$ of this problem for any $T>0$ and moreover by using the standard energy estimates, i.e. scalarly multiplying
\eqref{maxdyn} by $(e,h) \exp(-\gamma_0t)$
and integrating by parts over $\mathbb R^3 \times (0,t)$ we have
\begin{equation}\label{ebound}
\| e (t,\cdot)\|_{(0)}(\mathbb{R}^3) + \|h(t,\cdot)\|_{(0)}(\mathbb{R}^3 ) \leq C_0 \exp(\gamma_0 t),
\end{equation}
where positive $\gamma_0$ and $C_0$ might depend on $\epsilon, \mu, J$. Then the following Fourier-Laplace transforms are well defined
\begin{align}
\label{FourierLaplace}
E_*(x,\omega) = \frac{1}{\sqrt{2\pi}} \int_0^{\infty} e(t,x) \exp({i}\omega t) dt,\,
H_*(x,\omega) = \frac{1}{\sqrt{2\pi}} \int_0^{\infty} e(t,x) \exp({i}\omega t) dt
\end{align}
with $\omega=\omega_1+{i}\gamma$, $\gamma_0 < \gamma$.

Approximating $J_\ep$, $J_{\mu}$ by smooth functions and integrating by parts, we have
\begin{align*}
0 &= \int_0^{\infty} (\partial_t (\epsilon e) - \curl h )(t,\cdot) \exp({{i}\omega t}) dt \nonumber \\
 & = -\epsilon e(0,\cdot) -i\omega  \int_0^{\infty}\epsilon e(t,\cdot) \exp({{i}\omega t}) dt -\curl \int_0^{\infty} h(t,\cdot) \exp({{i}\omega t}) dt
\end{align*}
and
\begin{align*}
0 &= \int_0^{\infty} (\partial_t (\mu h) + \curl e )(t,\cdot) \exp({{i}\omega t}) dt \nonumber \\
 & = -\mu h(0,\cdot) -i\omega  \int_0^{\infty}\mu h(t,\cdot) \exp({{i}\omega t}) dt +\curl \int_0^{\infty} e(t,\cdot) \exp({{i}\omega t}) dt
\end{align*}
Hence $(E_*,H_*)$ solves \eqref{maxeq} with $\omega=\omega_1+{i} \gamma$, $\gamma_0 < \gamma$.
In addition, $(E_*,H_*)$ exponentially decays for large $|x|$. Indeed, due to the finite speed of the propagation in the hyperbolic problems we have $e(t,x)=0$ for $x\in \R^3\setminus B(R)$ if $t<\theta R -R_0$ for some $\theta=\theta(\ep,\mu)>0$, where $R_0>0$ satisfies $\bar\Omega\subset B(R_0)$ and $R>R_0$.  Hereafter, $B(R)$ denotes the ball of radius $R$ centered at $0$. Hence from
(\ref{FourierLaplace}), for any $\delta>0$
\begin{align*}
\int_{B(R+4)\setminus B(R)} |E_*|^2 & \leq \int_{B(R+4)\setminus B(R)} \left|\int_{\theta R-R_0}^{+\infty} e(t,x)\exp({i}\omega_1t-(\gamma-\delta)t) \exp({-\delta t}) dt \right|^2 dx \\
& \leq  \int_{B(R+4)\setminus B(R)} \left(\int_{\theta R-R_0}^{+\infty} |e(t,x)|^2\exp(-2(\gamma-\delta)t)dt  \int_{\theta R-R_0}^{+\infty} \exp(-2\delta t) dt \right) dx \\
& \leq C(C_0,\delta,R_0) \exp(-2\delta\theta R) \int_{\theta R-R_0}^{+\infty} \|e(t,\cdot)\|_{(0)}^2(\mathbb{R}^n\setminus B(R)) \exp(-2\gamma_0t)\exp(-2(\gamma-\gamma_0-\delta)t)dt \\
& \leq C(C_0,\delta,R_0) \exp(-2\delta\theta R) \int_{\theta R-R_0}^{+\infty}  \exp(-2(\gamma-\gamma_0-\delta)t)dt \\
& \leq C(C_0,\gamma, \gamma_0,R_0)  \exp(-2\delta \theta R),
\end{align*}
where we have used \eqref{ebound} and  chosen $\delta=\frac{\gamma-\gamma_0}{2}$. The same bound holds for $H_*$, and we yield
\begin{equation}
\label{expEH}  
\int_{B(R+4)\setminus B(R)} (|E_*|^2 + |H_*|^2) \leq 
 C(C_0,\gamma, \gamma_0,R_0)  \exp(-2\delta \theta R)
\end{equation}

By Theorem \ref{thm3} the vector function $(E(\,,\omega), H(\,,\omega))$ solving \eqref{maxeq},\eqref{rad} has a complex analytic extension from $(0,\infty)$ into a neighbourhood in the first quarter plane $\{0<\Re\omega, 0 \leq \Im{\omega}\}$, by the uniqueness of the analytic continuation this extension satisfies \eqref{maxeq},\ref{rad} and decays exponentially as $|x|\rightarrow \infty$ when $0 < \Im{\omega}$, in particular, we have for $E,H$ the bound \eqref{expEH}.  To show that $E=E_*, H=H_*$, we let  $E^0=E_*-E, H^0=H_*-H$. Since $(E,H), (E_*, H_*)$ solve the Maxwell system \eqref{maxeq} we see that $\curl E^0-i\omega\mu H^0=0,
\curl H^0+i\omega\ep E^0=0$ in $\mathbb{R}^3$. Let the cut off $C^1(\mathbb{R}^3)$- function $\varphi=1$ on $B(R)$,
$\varphi=0$ on $\mathbb{R}^3\setminus B(R+1)$, $0\leq \varphi \leq 1$.
 Using the homogeneous Maxwell equations for $(E^0,H^0)$, we obtain
 \begin{equation}
 \label{E0}
 \curl(\varphi E^0)-i\omega \mu \varphi H^0=\nabla\varphi\times E^0,\;
 \curl(\varphi H^0)+i\omega \ep \varphi E^0=\nabla\varphi\times H^0.
\end{equation}
By \eqref{E0}, integrating by parts, and using $\varphi=0$ on $\partial B(R+4)$ imply
\[
\begin{aligned}
0=&\int_{B(R+4)}( \curl(\varphi E^0)\cdot (\varphi \bar H^0)-
(\varphi E^0)\cdot \curl(\varphi \bar H^0))\\
=&\int_{B(R+4)}( i\omega (\varphi \mu H^0)\cdot (\varphi \bar H^0)-
(\varphi E^0)\cdot (i\bar\omega\varphi \ep \bar E^0))\\
&\;+
\int_{B(R+4)\setminus B(R)}(  (\nabla\varphi\times E^0)\cdot (\varphi \bar H^0)-
(\varphi E^0)\cdot (\nabla\varphi \times \bar H^0)).
\end{aligned}
\]
Therefore, taking the real part of the above relation yields
$$
\Im \omega\int_{B(R+4)}\varphi^2 (\ep E^0\cdot  \bar E^0+
\mu H^0\cdot \bar H^0)=
\Re \int_{B(R+4)\setminus B(R)}(  (\nabla\varphi\times E^0)\cdot (\varphi \bar H^0)-
(\varphi E^0)\cdot (\nabla\varphi \times \bar H^0))
$$
and hence from \eqref{expEH} and the exponential decay of
$E,H$ we derive
$$
\int_{B(R)}( E^0\cdot  \bar E^0+
 H^0\cdot \bar H^0) \leq C(C_0,\gamma,\gamma_0,R_0) exp(-\delta \theta R).
$$
Letting $R\rightarrow +\infty$ we conclude that $E^0=H^0=0$
on $\R^3$  and so 
 \begin{equation}
 \label{EE*}
 E(\,,\omega)= E^*(\,,\omega),\;\;
 H(\,,\omega)=H^*(\,,\omega)
 \end{equation}
when $\omega=\omega_1 +{i}\gamma$, $0<\omega_1$, and $\gamma_0 < \gamma$.

By Lemma \ref{lem22},  $E(\, ,\omega)=H(\, ,\omega)=0$ on $\R^3\setminus\Omega$, when $K_\ast < \omega < K$, and due to uniqueness of the analytic continuation with respect to $\omega$ in $\{0<\Re\omega, 0\le\Im\omega\}$, it follows from \eqref{EE*} that $E^\ast(\,,\omega)=H^\ast(\,,\omega)=0$ on $\partial\Omega$ for $\omega=\omega_1 +{i}\gamma$, $0<\omega_1$, $\gamma_0 < \gamma$. Now by the uniqueness in the Fourier-Laplace transform
\eqref{FourierLaplace}, we finally conclude that $e=h=0$ on $\partial \Omega \times (0,\infty)$.
In view of the structure assumption \eqref{structure}, by  the uniqueness of the Cauchy problem for the dynamical Maxwell system \eqref{maxdyn}, \eqref{compat}, we conclude that $e=h=0$ in $\Omega \times \{T_0\}$ for some (large) $T_0$ \cite{E03}.
Since $e=h=0$ in $\partial\Omega\times(0,T_0)$,  by uniqueness in the hyperbolic
(backward) initial boundary value problem with the initial data at $t=T_0$, we have $e=h=0$ in $\Omega \times (0,T_0)$. Hence $e=h=0$ on $\Omega \times \{0\}$ and from \eqref{initial} it follows that $J_\ep=J_\mu=0$. The proof is complete.
\end{proof}

\section{Quantitative analytic continuation}\label{sec3}

We will start preparations for a proof of Theorem~\ref{thm2}. Since $\ep=\ep_0, \mu=\ep_0$, \eqref{maxeq} becomes
\begin{equation}
\label{maxeq1}
\left\{
\begin{aligned}
\curl E-i\omega \mu_0 H&=J_\mu\quad\mbox{in}\quad\R^3,\\
\curl H+i\omega \ep_0 E&=J_\ep\quad\mbox{in}\quad\R^3,
\end{aligned}\right.
\end{equation}
with the radiation condition \eqref{rad} provided $\omega>0$. As in \eqref{intEH*}, these equations and the radiation condition are equivalent to the integral representation 
\begin{equation}\label{intEH}
\begin{aligned}
&E(x,\omega)=
\int_{\Omega}\frac{\exp({i\kappa|x-y|})}{4\pi|x-y|}(i\omega \mu_0 J_\ep(y)+\curl J_\mu(y))dy,\\
&H(x,\omega)=\int_{\Omega}\frac{\exp({i\kappa|x-y|})}{4\pi|x-y|}(-i
\omega \ep_0 J_\mu(y)+\curl J_\ep(y))dy,
\end{aligned}
\end{equation}
where $\kappa =\omega\sqrt{\ep_0\mu_0}$. Moreover, as follows from the standard representation of radiating solutions of the Helmholtz equations, \eqref{intEH} is equivalent
to the Helmholtz equations
\begin{equation}
\label{Helm}
\Delta E+\kappa^2E=-i\omega \mu_0 J_\ep-\curl J_\mu,\;
\Delta H+\kappa^2H=i\omega \ep_0 J_\mu-\curl J_\ep \quad\mbox{in}\quad\R^3
\end{equation}
and the Sommerfeld radiation condition.

In particular, we have 
\begin{equation}\label{te}
E(x,\omega)\times\nu(x)=\int_{\Omega}\frac{\exp({i \kappa |x-y|})}{4\pi|x-y|}(i\omega \mu_0 J_\ep(y)+\curl J_\mu(y))\times\nu(x)dy, \quad x\in\partial\Omega,
\end{equation}
and 
\begin{equation}\label{int}
E(x,\omega)\times\nu(x)-\alpha(x)H_\tau(x,\omega)\\
=\int_{\Omega}\frac{\exp({i \kappa |x-y|})}{4\pi|x-y|}G(x,y,\omega)dy,\quad x\in\partial\Omega,
\end{equation}
where
\[
\begin{aligned}
G(x,y,\omega)=&(i\omega \mu_0 J_\ep(y)+\curl J_\mu(y))\times\nu(x)\\
&-\alpha(x)\{(-i\omega\ep_0 J_\mu(y)+\curl J_\ep(y))-[(-i\omega\ep_0 J_\mu(y)+\curl J_\ep(y))\cdot\nu(x)]\nu(x)\}.
\end{aligned}
\]
Observe that the formulae on the right sides of \eqref{te} and \eqref{int} can be defined even when $\omega<0$. Therefore, for $\omega<0$, we define $E(x,\omega)\times\nu(x)$ and  $E(x,\omega)\times\nu(x)-\alpha(x)H_\tau(x,\omega)$ in terms of formulae \eqref{te} and \eqref{int}, respectively. 

Now it follows from  (\ref{int}) that
\begin{equation}
\label{int0}
\int_{-\infty}^{\infty}\|E\times\nu(,\omega)-\alpha H_\tau(,\omega)\|_{(0)}^2(\partial\Omega)d\omega  = I_0(k) + \int_{k< |\omega|} \|E\times\nu(,\omega)-\alpha H_\tau(,\omega)\|_{(0)}^2(\partial\Omega)d\omega, 
\end{equation}
where $I_0(k)$ is defined as
\begin{equation}
I_0(k) = 2\int_0^k \int_{\partial\Omega}
\left(\int_\Omega 
\frac{\exp({i\kappa|x-y|})}{4\pi|x-y|}G(x,y,\omega) dy\right)\cdot
\left(\int_\Omega 
\frac{\exp({-i\kappa|x-y|})}{4\pi|x-y|} G(x,y,-\omega)  dy\right) d \Gamma(x) d\omega. \label{I0}
\end{equation}
As above, we write
\begin{equation}
\label{int1}
\int_{-\infty}^{\infty}\|E\times\nu(,\omega)\|_{(1)}^2(\partial\Omega)d\omega = I_1(k) + \int_{k< |\omega|}\|E\times\nu(,\omega)\|_{(1)}^2(\partial\Omega)d\omega,
\end{equation}
where
\[
I_1(k)=2\int_0^k \int_{\partial\Omega}\left(|(E\times\nu)(x,\omega)|^2+|\nabla_{\partial\Omega}(E\times\nu)(x,\omega)|^2\right)d\Gamma(x)d\omega.
\]
We observe that $\nabla E$ is viewed as the vector with
9 components ($\partial_j E_k$) and
$\nabla_{\partial\Omega} E$ is the tangential projection of the gradient (the 9 dimensional vector formed of tangential projections of gradients of 3 components of $E$).
We have
$$
|(E\times\nu)(x,\omega)|^2=
(E\times\nu)(x,\omega)\cdot\overline{(E\times\nu)}(x,\omega)=
(E\times\nu)(x,\omega)\cdot(E\times\nu)(x,-\omega),
$$
due to \eqref{te}. Remind that we assume that $J_\mu$ and $J_\ep$ are real-valued. Similarly,
$$
|\nabla_{\partial\Omega}(E\times\nu)(x,\omega)|^2=
\nabla_{\partial\Omega}(E\times\nu)(x,\omega)\cdot
\nabla_{\partial\Omega}(E\times\nu)(x,-\omega),
$$
and hence, again using \eqref{te}, the integrand in
$I_1(k)$ can be extended to an entire analytic function of $\omega$. 
Hence
\begin{equation}
\label{I1}
I_1(k)=2\int_0^k \int_{\partial\Omega}\big{(}
(E\times\nu)(x,\omega)\cdot(E\times\nu)(x,-\omega) +
\nabla_{\partial\Omega}(E\times\nu)(x,\omega)\cdot
\nabla_{\partial\Omega}(E\times\nu)(x,-\omega)\big{)} d\Gamma(x)
d\omega.
\end{equation}
Since (due to \eqref{te}, \eqref{int}, \eqref{I0}, \eqref{I1})  the integrands are entire analytic functions of $\omega=\omega_1+i\omega_2$, we can analytically extend $I_0(k)$ and $I_1(k)$ from $k>0$ to $k\in{\mathbb C}$ and, moreover, the integrals in \eqref{I0}, \eqref{I1} with respect to $\omega$ can be taken over any path joining points $0$ and $k=k_1+ik_2$ of the complex plane. Thus $I_0(k)$ and $I_1(k)$ are entire analytic functions of $k\in{\mathbb C}$.  

Due to the definitions of the norms of the boundary data
\begin{equation}\label{varepsilon}
\varepsilon_0^2 = \frac{I_0(K)}{2}\quad\mbox{and}\quad\varepsilon_1^2=\frac{I_1(K)}{2}.
\end{equation}
 The truncation level $k$ in (\ref{int0}) and \eqref{int1} is important to keep balance between the known data and the unknown information when $k\in [K,\infty)$.

We will need the following elementary estimate for $I_0(k)$.
\begin{lemma}
Let $\supp J_\ep, \supp J_\mu \subset \Omega$, then for $k=k_1+ik_2$
\begin{equation}
\label{boundI0}
|I_0(k)|\leq
C (1+|k|^3) 
(\|J_\ep\|^2_{(1)}(\Omega)+\|J_\mu\|^2_{(1)}(\Omega))
\exp({2 d \sqrt{\ep_0\mu_0}|k_2|}),
\end{equation}
where  $ d= \sup |x-y|$ over $x,y \in \Omega$.
\end{lemma}

\begin{proof}
Using the parametrization $\omega = ks, s \in (0,1)$ in the line integral and the elementary inequality
$|\exp({i\sqrt{\ep_0\mu_0}\omega|x-y|})|\leq \exp({\sqrt{\ep_0\mu_0}|k_2|d })$
it is easy to derive from \eqref{I0}, \eqref{int} that
\begin{align*}
&|I_0(k)|\\
\leq&
C\int_0^1|k| 
\left(\int_{\partial\Omega}
\left(\int_{\Omega}\{|k|s (|J_\ep(y)|+|J_\mu(y)|)+|\curl J_\ep(y)|+|\curl J_\mu(y)|\}
\frac{\exp({\sqrt{\ep_0\mu_0}|k_2| d})}{|x-y|} dy\right)^2
 d \Gamma(x) \right) d s \\
\leq&
C|k|
\int_{\partial\Omega}
\left(\int_{\Omega} \{|k|^2 (|J_\ep(y)|^2+|J_\mu(y)|^2)+|\curl J_\ep(y)|^2+|\curl J_\mu(y)|^2\} dy\right)
\left(\int_{\Omega} \frac{\exp({2\sqrt{\ep_0\mu_0}|k_2|d})}{|x-y|^2} dy\right)
 d \Gamma(x),
\end{align*}
where the Cauchy inequality is used for the integrals with respect to $y$. Since 
$$
\int_{\Omega} \frac{1}{|x-y|^2} dy \leq C,
$$
we yield
$$
|I_0(k)| \leq 
C|k|\left(\int_{\Omega}\{|k|^2(|J_\ep(y)|^2+|J_\mu(y)|^2)+|\nabla J_\ep(y)|^2+|\nabla J_\mu(y)|^2\} dy\right)\exp({2\sqrt{\ep_0\mu_0}|k_2|d}) 
$$
and complete the proof of (\ref{boundI0}).

\end{proof}

\begin{lemma}
Let $\supp J_\ep, \supp J_\mu \subset \Omega$, then for $k=k_1+ik_2$
\begin{equation}
\label{boundI1}
|I_1(k)|\leq
C (1+|k|^3)
(\|J_\ep\|^2_{(2)}(\Omega)+\|J_\mu\|^2_{(2)}(\Omega))
\exp({2\sqrt{\ep_0\mu_0}|k_2|d}),
\end{equation}
where  $ d= \sup |x-y|$ over $x,y \in \Omega$.
\end{lemma}

\begin{proof}

We $C^1$ extend the vector field $\nu$ from $\partial\Omega$ onto some neighbourhood  $V$ of $\partial\Omega$ and denote this extension again as $\nu$.
Using the parametrization 
$\omega=ks, 0\leq s\leq 1,$ in the integral
\eqref{I1} we have
\begin{equation}\label{int1anal}
\begin{aligned}
&|I_1(k)|\\
=&2|k|\int_0^1 \int_{\partial\Omega}
\left(|(E\times\nu)(x,\omega)||(E\times\nu)(x,-\omega)|+
|\nabla_{\partial\Omega}(E\times\nu)(x,\omega)|| \nabla_{\partial\Omega}(E\times\nu)(x,-\omega)|\right)d\Gamma(x)d s\\
\le&2|k|\int_0^1 \int_{\partial\Omega}
\left(|E\times\nu)(x,\omega)||(E\times\nu)(x,-\omega)|+
|\nabla (E\times\nu)(x,\omega)|| \nabla (E\times\nu)(x,-\omega)|\right)d\Gamma(x)d s
\end{aligned}
\end{equation}
when $k=k_1+ik_2$.  

From \eqref{te}, by  the Cauchy inequality, it follows that
$$
|(E\times\nu)(x,\omega)|^2\leq
C \exp (2\sqrt{\ep_0\mu_0}|k_2|d)
\left(\int_{\Omega} \frac{1}{|x-y|^2} dy\right)\int_{\Omega}(|\omega|^2|J_{\ep}|^2+|\nabla J_{\mu}|^2)(y) dy.
$$

Using
$$
\frac{\partial}{\partial x_j}|x-y|= -\frac{\partial}{\partial y_j}|x-y|
$$
and integrating by parts with respect to $y$ imply
\[
\begin{aligned}
&\frac{\partial}{\partial x_j}\int_{\Omega}
\frac{\exp({i\kappa|x-y|})}{|x-y|}
(i\omega\mu_0J_{\ep}(y)+ \curl J_{\mu}(y))\times \nu(x) dy\\
=&
\int_{\Omega}
\frac{\exp({i\kappa |x-y|})}{|x-y|}
\frac{\partial}{\partial y_j}(i\omega\mu_0 J_{\ep}(y)+ \curl J_{\mu}(y))\times \nu(x) dy\\
&+
\int_{\Omega}
\frac{\exp({i\kappa|x-y|})}{|x-y|}
(i\omega\mu_0J_{\ep}(y)+ \curl J_{\mu}(y))\times \frac{\partial}{\partial x_j}\nu(x) dy.
\end{aligned}
\]
Hence, from \eqref{te} and using the Cauchy inequality for the integrals with respect to $y$, we derive that
\[
\begin{aligned}
&|\nabla (E\times \nu)(x,\omega)|^2\leq C
\exp({2\sqrt{\ep_0\mu_0}|k_2|d})\\
&\left(\int_{\Omega}\frac{1}{|x-y|^2} dy \right)\int_{\Omega} ((1+|k|^2) (
|J_{\ep}|^2+|J_{\mu}|^2 +|\nabla J_{\ep}|^2+|\nabla J_{\mu}|^2)
+(|\nabla^2 J_{\ep}|^2+|\nabla^2 J_{\mu}|^2) dy\\
&\leq
C \exp({2\sqrt{\ep_0\mu_0}|k_2|d}) (|k|^2+1) 
(\|J_{\ep}\|_{(2)}^2(\Omega)+\|J_{\mu}\|_{(2)}^2(\Omega)),
\end{aligned}
\]
and combining with \eqref{int1anal} complete the proof of \eqref{boundI1}.

\end{proof}

The following steps are needed to link the unknown values of $I_0(k)$ for $k\in [K,\infty)$  to the known values $\varepsilon_0,
\varepsilon_1$  in \eqref{varepsilon}. Let $S$ be the sector $\{k: -\frac{\pi}{4} < \arg k < \frac{\pi}{4} \}$ and $\mu(k)$
be the harmonic measure of the interval
$[0,K]$ in $S\setminus [0,K]$. Observe that $|k_2|\leq k_1$
( and hence $|k|\leq 2k_1$) when $k\in S$, so from \eqref{boundI0} we have
 $$
\begin{aligned}
&|I_0(k)\exp({-2(d+1)\sqrt{\ep_0\mu_0}k})|\\
\leq&
C  (1+k_1^3)
(\|J_\ep\|^2_{(1)}(\Omega)+\|J_\mu\|^2_{(1)}(\Omega))
\exp(2\sqrt{\ep_0\mu_0}d k_1)\exp(- 2(d+1)\sqrt{\ep_0\mu_0}k_1)\\
\leq&C((1+k_1^3)M_1^2 \exp(- 2\sqrt{\ep_0\mu_0}k_1)\leq C M_1^2
\end{aligned}
$$
with  generic constants $C$.
Due to \eqref{varepsilon},
$$
|I_0(k)\exp({-2(d+1)\sqrt{\ep_0\mu_0}k})|\leq 2\varepsilon_0^2\;\;\text{when}\;\; k\in [0,K],
$$
so as in \cite[page 55, Theorem 2]{LRS} and \cite[page 67]{I}, we conclude that when $K<k<+\infty$
\begin{equation}
\label{I0omega}
|I_0(k)\exp({-2(d+1)\sqrt{\ep_0\mu_0}k})|\leq
C \varepsilon_0^{2\mu(k)}M_1^2.
\end{equation}
Using \eqref{boundI1} instead of \eqref{boundI0} and carrying out the same computations as above, we can obtain that 
for $K<k<+\infty$
\begin{equation}
\label{I1omega}
|I_1(k)\exp({-2(d+1)\sqrt{\ep_0\mu_0}k})|\leq
C \varepsilon_1^{2\mu(k)}M_2^2.
\end{equation}

We need  the following lower bound of the harmonic measure $\mu(k)$ given in \cite{CIL}, Lemma 2.2.

\begin{lemma}
If $0<k< 2^{\frac{1}{4}}K$, then
\begin{equation*}
\frac{1}{2}\leq \mu(k).
\end{equation*}
On the other hand, if $ 2^{\frac{1}{4}}K < k$, then
\begin{equation}
\label{boundharm2}
\frac{1}{\pi}\left(\left(\frac{k}{K}\right)^4-1\right)^{-\frac{1}{2}} \leq \mu(k).
\end{equation}
\end{lemma}

\section{Time-dependent Maxwell and wave equations}\label{sec4}

 Let $e$ be a solution to the initial value problem of the wave equation:
\begin{equation}\label{tw}
\left\{
\begin{aligned}
&\ep_0\mu_0\partial^2_t e-\Delta e=0\quad\mbox{in}\quad\mathbb R^3\times(0,\infty),\\
&e(x,0)=-\frac{\sqrt{2\pi}}{\ep_0}J_\ep,\quad \partial_te(x,0)=\frac{\sqrt{2\pi}}{\ep_0 \mu_0} \curl J_\mu\quad\mbox{on}\quad\mathbb R^3\times\{0\},
\end{aligned}\right.
\end{equation}
and $h$ be a solution to the initial value problem
\begin{equation}\label{th}
\left\{
\begin{aligned}
&\ep_0\mu_0\partial^2_t h-\Delta h=0\quad\mbox{in}\quad\mathbb R^3\times(0,\infty),\\
&h(x,0)=\frac{\sqrt{2\pi}}{\mu_0}J_\mu,\quad \partial_t h(x,0)=\frac{\sqrt{2\pi}}{\ep_0 \mu_0}\curl J_\ep\quad\mbox{on}\quad\mathbb R^3\times\{0\}.
\end{aligned}\right.
\end{equation}
Observe that if $\mbox{\rm div}\, J=0$, then $\mbox{\rm div}\, e=0$ for all $t>0$.

As shown in \cite{CIL}, \eqref{Helm} implies that
\[
E(x,\omega)=\frac{1}{\sqrt{2\pi}}\int_0^{+\infty}e(x,t)\exp({i\omega t})dt,\;\; H(x,\omega)=\frac{1}{\sqrt{2\pi}}\int_0^{+\infty}h(x,t)\exp({i\omega t})dt,
\]
Setting $e(x,t)=h(x,t)=0$ for $t<0$, we can write
\begin{equation}
\label{Ee}
E(x,\omega)=\frac{1}{\sqrt{2\pi}}\int_{-\infty}^{+\infty}e(x,t)\exp({i\omega t})dt,\;\; H(x,\omega)=\frac{1}{\sqrt{2\pi}}\int_{-\infty}^{+\infty}h(x,t)\exp({i\omega t})dt.
\end{equation}.

To proceed, we need to estimate the remainders in \eqref{int0},
\eqref{int1}. We first prove the next result, which is similar to \cite[Lemma~4.1]{CIL} and \cite[Lemma~2.3]{EI18}.

\begin{lemma}
\label{lemma_apriori}
Let $(E,H)$ be the electric and magnetic fields satisfying \eqref{maxeq1}, \eqref{rad} with $\supp J_\ep, \supp J_\mu \subset \Omega$. Then if $J_\ep, J_\mu \in H^1(\Omega)$ satisfy \eqref{m0}, we have 
\begin{equation}
\label{apriori0}
\int_{k<|\omega|}
\|E\times\nu( ,\omega)-\alpha H_\tau(,\omega)\|^2_{(0)}(\partial\Omega) d\omega \leq C k^{-2}(\|J_\ep\|^2_{(1)}(\Omega)+\|J_\mu\|^2_{(1)}(\Omega))\le Ck^{-2}M_1^2.
\end{equation}
On the other hand, if $J_\ep, J_\mu \in H^2(\Omega)$and \eqref{m1} holds, then
\begin{equation}
\label{apriori1}
\int_{k<|\omega|}
\|E\times\nu( ,\omega)\|^2_{(1)}(\partial\Omega) d\omega \leq C k^{-2}(\|J_\ep\|^2_{(2)}(\Omega)+\|J_\mu\|^2_{(2)}(\Omega))\le Ck^{-2}M_2^2.
\end{equation}
\end{lemma}

\begin{proof}
We first prove \eqref{apriori0}. By Plancherel's formula, we have that
\begin{equation}\label{318}
\begin{aligned}
&\int_{k<|\omega|}
\|E\times\nu( ,\omega)-\alpha H_\tau(,\omega)\|^2_{(0)}(\partial\Omega)
d\omega\\
\leq& k^{-2}
\int_{k<|\omega|}
\omega^2\|E\times\nu( ,\omega)-\alpha H_\tau(,\omega)\|^2_{(0)}(\partial\Omega)
d\omega \\
\leq&
k^{-2} \int_{\mathbb R}
\omega^2\|E\times\nu( ,\omega)-\alpha H_\tau(,\omega)\|^2_{(0)}(\partial\Omega)
d\omega\\
= & k^{-2}
\int_{\mathbb R}
\|\partial_t e( ,t)\times\nu-\partial_t(\alpha h_\tau(,t))\|^2_{(0)}(\partial\Omega)
dt\\
\le&C k^{-2}
\int_{\mathbb R}
(\|\partial_t e( ,t)\|^2_{(0)}(\partial\Omega)+\|\partial_t h( ,t)\|^2_{(0)}(\partial\Omega))
dt.
\end{aligned}
\end{equation}
Combining the Huygens' principle 
\begin{equation}\label{huygen}
e( ,t)=h( ,t)=0\;\;\mbox{on}\;\;\Omega,\;\;\mbox{when}\;\;\sqrt{\ep_0\mu_0} d<t, 
\end{equation}
and the following estimate
\begin{equation*}
\|e\|^2_{(1)}
(\partial\Omega \times (0,\sqrt{\ep_0\mu_0}d))+\|h\|^2_{(1)}
(\partial\Omega \times (0,\sqrt{\ep_0\mu_0}d)) \leq C(\|J_\ep\|^2_{(1)}(\Omega)+\|J_\mu\|^2_{(1)}(\Omega)),
\end{equation*}
which follows from the generalization \cite{I1} of Sakamoto energy estimates \cite{S} to the transmission problems (see also \cite[(2.31)]{EI18}), \eqref{apriori0} is an easy consequence of \eqref{318}.

To prove \eqref{apriori1}, it follows from the similar argument that
\[
\int_{k<|\omega|}
\|E\times\nu( ,\omega)\|^2_{(1)}(\partial\Omega)
d\omega\le C k^{-2}
\int_{\mathbb R}
(\|\partial_t e( ,t)\|^2_{(0)}(\partial\Omega)+
\|\partial_t \nabla e( ,t)\|^2_{(0)}(\partial\Omega))dt.
\]
By the Huygens's principle and the above generalization of Sakamoto energy estimates applied to $\partial_j e$, we have
\[
\|\partial_t\nabla e\|^2_{(0)}(\partial\Omega\times(0,\sqrt{\ep_0\mu_0}d))\le C(\|J_\ep\|^2_{(2)}(\Omega)+\|J_\mu\|^2_{(2)}(\Omega))
\]
(see \cite[(4.20)]{CIL}). \eqref{apriori1} follows easily. The proof is complete. 

\end{proof}

Now we consider the initial value problem for the time-dependent Maxwell equations
\begin{equation}\label{tM}
\left\{
\begin{aligned}
&\mu_0\partial_t h^*+ \curl e^*=0\quad\mbox{in}\quad\mathbb R^3\times(0,\infty),\\
&\ep_0 \partial_t e^*- \curl h^*=0\quad\mbox{in}\quad\mathbb R^3\times(0,\infty),\\
&h^*( ,0)=\frac{\sqrt{2\pi}}{\mu_0}J_\mu,\quad e^*( ,0)=-\frac{\sqrt{2\pi}}{\ep_0}J_\ep\quad\mbox{on}\quad\mathbb R^3\times\{0\}.
\end{aligned}\right.
\end{equation}
Since $\Div J_\ep=0=\Div J_\mu$, by the uniqueness of the initial value problem, we have $\Div e^*( ,t)=0=\Div h^*( ,t)$. So from \eqref{tM} we obtain
$$
\ep_0\mu_0\partial_t^2 e^*= \mu_0 \curl \partial_t h^*=
-\curl \curl e^*=\Delta e^*
$$
and similarly 
$$
\ep_0\mu_0\partial_t^2 h^*=\Delta h^*.
$$
Also it is easy to see that $e^*, h^*$ satisfy the same initial conditions \eqref{tw}, \eqref{th}.
Using the uniqueness in the initial value problem for the wave equations we derive that
\begin{equation}
\label{ee*}
e=e^*,\; h=h^*.
\end{equation}

From \eqref{ee*}, Maxwell system \eqref{tM}, and \eqref{huygen} we have
\begin{equation*}
\left\{
\begin{aligned}
&\mu_0\partial_t h+ \curl e=0\quad\mbox{in}\quad\Omega\times(0,T),\\
&\ep_0 \partial_t e- \curl h=0\quad\mbox{in}\quad\Omega\times(0,T),\\
&h( ,T)=0,\quad e( ,T)= 0\quad\mbox{on}\quad\Omega,
\end{aligned}\right.
\end{equation*}
where $T=\sqrt{\ep_0\mu_0}d$. For this backward initial value problem, using the estimates of Proposition 1.1 in \cite{CE} (for $e\times\nu-\alpha h_\tau$) or Corollary~1.4 in \cite{E08} (for $e\times\nu$), we can obtain the following key energy bounds.
\begin{lemma}
\label{lem_observ}
Let $(e,h)$ be a solution to \eqref{tM} with $J_\ep, J_\mu\in [L^2(\Omega)]^3$ and $\supp J_\ep, \supp J_\mu\subset \Omega$. Then
\begin{equation}\label{exact0}
\|J_\ep\|^2_{(0)}(\Omega)+\|J_\mu\|^2_{(0)}(\Omega) \leq
C\|e\times\nu-\alpha h_\tau\|^2_{(0)}(\partial\Omega\times (0,T))
\end{equation}
and
\begin{equation}\label{exact1}
\|J_\ep\|^2_{(0)}(\Omega)+\|J_\mu\|^2_{(0)}(\Omega) \leq
C\|e\times\nu\|^2_{(1)}(\partial\Omega\times (0,T)).
\end{equation}
\end{lemma}

Now we are ready to prove the increasing stability results \eqref{stability}, \eqref{stability2} of Theorem \eqref{thm2}. 

\begin{proof}

We will first prove \eqref{stability} by modifying the argument in \cite{CIL} and \cite{EI18}. We may assume that $\varepsilon_0< \frac 12$, otherwise, \eqref{stability} obviously holds. In \eqref{int0} and \eqref{varepsilon}, we choose
\[
\left\{
\begin{aligned}
&k= \delta K^{\frac 23 } {\cal E}_0^{\frac 13}, \delta =(2\pi (d+1)\sqrt{\ep_0\mu_0})^{-\frac 13}\quad\mbox{when}\quad 2^{\frac 34}\delta^{-3}K< {\cal E}_0,\\
&k=K\quad\mbox{when}\quad {\cal E}_0 \leq 2^{\frac 34}\delta^{-3}K.
\end{aligned}\right .
\]
Now if $2^{\frac 34 }\delta^{-3}K< {\cal E}_0$, then $2^{\frac 14 }K<k$, and \eqref{I0omega}, \eqref{boundharm2} imply that
\[
\begin{aligned}
|I_0(k)| &\le CM_1^2\exp(2(d+1)\sqrt{\ep_0\mu_0}k)\exp\left(-\frac{2}{\pi}((\frac kK)^4-1)^{-\frac 12} {\cal E}_0\right)\\
&\leq CM_1^2 \exp(2(d+1)\sqrt{\ep_0\mu_0}k)\exp\left(-\frac{2}{\pi}(\frac Kk)^2 {\cal E}_0\right)\\
&=CM_1^2 \exp \left(2(d+1)\sqrt{\ep_0\mu_0}k-\frac{2}{\pi}\delta \delta^{-3}K^{\frac 23} {\cal E}_0^{\frac 13}\right)\\
&=CM_1^2 \exp \left(-2(d+1)\sqrt{\ep_0\mu_0}\delta K^{\frac 23} {\cal E}_0^{\frac 13}\right).
\end{aligned}
\]
Using the inequality $\exp(-y)\le 2y^{-2}$ when $y>0$, we have that
\begin{equation}\label{ineq00}
|I_0(k)|\le  \frac{CM_1^2}{K^{\frac 43} {\cal E}_0^{\frac 23}}.
\end{equation}
Next if $ {\cal E}_0 \leq 2^{\frac 34 }\delta^{-3}K$, then $k=K$ and by \eqref{varepsilon}
\begin{equation}
\label{ineq01}
|I_0(k)|=|I_0(K)|=2\varepsilon_0^2,\quad \frac{1}{K^{2}}\leq\frac{C}{K^{\frac 43}{\cal E}_0^{\frac 23}},
\end{equation}
since ${\cal E}_0\leq 2^{\frac 34 }\delta^{-3}K$.

Therefore, from \eqref{exact0} and  the Plancherel's formula we can estimate
\[
\begin{aligned}
&\|J_\ep\|^2_{(0)}(\Omega)+\|J_\mu\|^2_{(0)}(\Omega)\\
\leq&
C\|e\times\nu-\alpha h_\tau\|^2_{(0)}(\partial\Omega\times (0, \sqrt{\ep_0\mu_0}d))\le C\|e\times\nu-\alpha h_\tau\|^2_{(0)}(\partial\Omega\times\R)\\
=&C\int_{-\infty}^{\infty}\|E\times\nu(,\omega)-\alpha H_\tau(,\omega)\|_{(0)}^2(\partial\Omega)d\omega \leq C\varepsilon_0^2+\frac{CM_1^2}{K^{\frac 43}{\cal E}_0^{\frac 23}},
\end{aligned}
\]
when we consider the cases $2^{\frac 34 }\delta^{-3}K \leq {\cal E}_0$, $ {\cal E}_0 \leq 2^{\frac 34 }\delta^{-3}K$ and  use  \eqref{int0}, \eqref{apriori0} with our choice of $k$ in the both cases and the inequalities \eqref{ineq00}, \eqref{ineq01}. Since we assumed that $1< K, \varepsilon_0<\frac 12$, \eqref{stability} follows.

A proof of \eqref{stability2} can be obtained by a slight change in the previous argument. We may assume that $\varepsilon_1< \frac 12$, otherwise, \eqref{stability2} is obvious. In \eqref{int1} and \eqref{varepsilon}, we choose
\[
\left\{
\begin{aligned}
&k= \delta K^{\frac 23 } {\cal E}_1^{\frac 13}, \delta =(2\pi (d+1)\sqrt{\ep_0\mu_0})^{-\frac 13}\quad\mbox{when}\quad 2^{\frac 34}\delta^{-3}K< {\cal E}_1,\\
&k=K\quad\mbox{when}\quad {\cal E}_1 \leq 2^{\frac 34}\delta^{-3}K.
\end{aligned}\right .
\]
Now if $2^{\frac 34 }\delta^{-3}K< {\cal E}_1$, then $2^{\frac 14 }K<k$, and \eqref{I1omega}, \eqref{boundharm2} imply that
\[
\begin{aligned}
|I_1(k)| &\le CM_2^2\exp(2(d+1)\sqrt{\ep_1\mu_0}k)\exp\left(-\frac{2}{\pi}((\frac kK)^4-1)^{-\frac 12} {\cal E}_1\right)\\
&\leq CM_2^2 \exp \left(2(d+1)\sqrt{\ep_0\mu_0}k-\frac{2}{\pi}\delta \delta^{-3}K^{\frac 23} {\cal E}_1^{\frac 13}\right)\\
&=CM_2^2 \exp \left(-2(d+1)\sqrt{\ep_0\mu_0}\delta K^{\frac 23} {\cal E}_1^{\frac 13}\right),
\end{aligned}
\]
hence as above
\begin{equation}\label{ineq10}
|I_1(k)|\le  \frac{CM_2^2}{K^{\frac 43} {\cal E}_1^{\frac 23}}.
\end{equation}
If $ {\cal E}_1 \leq 2^{\frac 34 }\delta^{-3}K$, then $k=K$ and by \eqref{varepsilon}
\begin{equation}
\label{ineq11}
|I_1(k)|=|I_1(K)|=2\varepsilon_1^2,\quad \frac{1}{K^{2}}\leq\frac{C}{K^{\frac 43}{\cal E}_1^{\frac 23}},
\end{equation}
since ${\cal E}_1\leq 2^{\frac 34 }\delta^{-3}K$. Therefore, from \eqref{exact1} and  the Plancherel's formula we can estimate
\[
\begin{aligned}
&\|J_\ep\|^2_{(0)}(\Omega)+\|J_\mu\|^2_{(0)}(\Omega)
\leq
C\|e\times\nu\|^2_{(1)}(\partial\Omega\times (0, \sqrt{\ep_0\mu_0}d))\\
&\le C\int_{-\infty}^{\infty}\|E\times\nu(,\omega)\|_{(1)}^2(\partial\Omega)d\omega \leq C\varepsilon_1^2+\frac{CM_1^2}{K^{\frac 43}{\cal E}_1^{\frac 23}},
\end{aligned}
\]
when we consider the cases $2^{\frac 34 }\delta^{-3}K \leq {\cal E}_1$, $ {\cal E}_1 \leq 2^{\frac 34 }\delta^{-3}K$ and  use  \eqref{int1}, \eqref{apriori1} with our choice of $k$ in the both cases and the inequalities \eqref{ineq10}, \eqref{ineq11}. Since we assumed that $1< K, \varepsilon_1<\frac 12$, \eqref{stability2} follows.

\end{proof}

\section{Conclusion and discussion}

In this work, we study the inverse source problem for electromagnetic waves using the measurements involving tangential components of electric and magnetic fields at many frequencies. For the uniqueness, we consider a rather general setting in which the media are anisotropic and inhomogeneous. We measure the tangential components of electric and magnetic fields on a part of boundary for frequency $\omega\in (0,K)$. Under the structure assumption on the electric permittivity $\ep$ and permeability $\mu$ \eqref{structure}, the uniqueness is established for divergence-free sources. Since we use the Fourier transform in time to reduce our inverse source problem to identification of the initial data in the time-dependent Maxwell equations by data on the lateral boundary, the structure assumption  \eqref{structure} is needed to guarantee the uniqueness of the lateral Cauchy problem for time-dependent Maxwell equations. We want to point out that for the time harmonic Maxwell equations with general anisotropic media (without any structure assumption), the unique continuation property holds (see Lemma~\ref{lem22}). Proving the uniqueness of the lateral Cauchy problem for the time-dependent Maxwell equations with general anisotropic media remains an open problem. 

Our second result is the increasing stability of identifying sources using the $L^2$ norm of the absorbing boundary data $E(\,,\omega)\times\nu-\alpha H_\tau(\,,\omega)$ and the $H^1$ norm of the tangential of the electric field $E(\,,\omega)\times\nu$ for $\omega\in(0,K)$. It is tempting to prove the increasing stability using the $L^2$ norm of $E(\,,\omega)\times\nu$ for $\omega\in(0,K)$. However, since such boundary condition does not satisfy the Kreiss-Sakamoto condition, this task will be quite challenging. Finally, in the proofs of Theorem~\ref{thm1} and \ref{thm2}, it is crucial to assume that sources are the divergence-free. Due to the pedagogical example given in the introduction, it seems necessary to impose this restriction on sources. Therefore, to what extent one can determine a nondivergence-free source by boundary data at many frequencies is an interesting question.

\section*{Acknowledgements}
Isakov is supported in part by the Emylou Keith and Betty Dutcher Distinguished Professorship and the NSF grants DMS 15-14886 and
DMS 20-08154. Wang is partially supported by MOST 108-2115-M-002-002-MY3 and 109-2115-M-002-001-MY3.

\end{document}